\documentclass[12pt]{amsart}
\usepackage[sans]{dsfont}
\usepackage{amsfonts,amssymb,amsbsy,amsmath,amsthm,enumerate}

\numberwithin{equation}{section}

\newtheorem{theorem}{Theorem}[section]

\newtheorem{lemma}[theorem]{Lemma}

\newtheorem{pr}[theorem]{Proposition}
\theoremstyle{definition}

\newcommand{\bel}{\begin{equation} \label}
\newcommand{\ee}{\end{equation}}
\newcommand{\one}{\mathds{1}}

\newcommand{\supp}{{\text{supp}}}

\newcommand{\rd}{{\mathbb R}^{2}}
\newcommand{\re}{{\mathbb R}}
\newcommand{\R}{{\mathbb R}}

\newcommand{\Z}{{\mathbb Z}}
\newcommand{\N}{{\mathbb N}}

\newcommand{\T}{{\mathcal T}}
\newcommand{\J}{{\mathcal J}}
\newcommand{\Td}{{\mathcal T}^2}

\newcommand{\eps}{{\epsilon}}

\def\beq{\begin{equation}}
\def\eeq{\end{equation}}
\newcommand{\bea}{\begin{eqnarray}}
\newcommand{\eea}{\end{eqnarray}}
\newcommand{\beas}{\begin{eqnarray*}}
\newcommand{\eeas}{\end{eqnarray*}}

\DeclareMathOperator*{\esssup}{ess\,sup}
\newcommand{\dth}{\dot{\theta}}
\newcommand{\domth}{{\mathcal M}_\theta}
\newcommand{\dom}{M}
\newcommand{\doml}{M_\ell}
\newcommand{\cs}{m}
\newcommand{\ze}{\mathbb Z}
\newcommand{\nft}{\|\partial_\tau \varphi_1\|_{{\rm L}^2(\cs)}}
\DeclareMathOperator*{\essinf}{ess\,inf}
\usepackage[normalem]{ulem}
\usepackage{color}

\newcommand\soutD{\bgroup\markoverwith
{\textcolor{blue}{\rule[.5ex]{2pt}{0.8pt}}}\ULon}
\newcommand{\Hm}[1]{\leavevmode{\marginpar{\tiny%
$\hbox to 0mm{\hspace*{-0.5mm}$\leftarrow$\hss}%
\vcenter{\vrule depth 0.1mm height 0.1mm width \the\marginparwidth}%
\hbox to 0mm{\hss$\rightarrow$\hspace*{-0.5mm}}$\\\relax\raggedright #1}}}

\begin{document}
\title[Lifshits tails for random quantum waveguides]{Lifshits tails for randomly twisted quantum waveguides}

\author[W.~Kirsch]{Werner Kirsch}
\author[D.~Krej\v{c}i\v{r}\'ik]{David Krej\v{c}i\v{r}\'ik}
\author[G.~Raikov]{Georgi Raikov}

\begin{abstract}\setlength{\parindent}{0mm}
We consider the Dirichlet Laplacian $H_\gamma$ on a 3D twisted waveguide with random Anderson-type twisting $\gamma$. We introduce the integrated density of states $N_\gamma$ for the operator $H_\gamma$, and investigate the Lifshits tails of $N_\gamma$, i.e. the asymptotic behavior of $N_\gamma(E)$ as $E \downarrow \inf {\rm supp}\, dN_\gamma$. In particular, we study the dependence of the Lifshits exponent on the decay rate of the single-site twisting at infinity.
\end{abstract}

\maketitle

{\bf  AMS 2010 Mathematics Subject Classification:} 82B44,  35R60, 47B80,  81Q10\\

{\bf  Keywords:} randomly twisted quantum waveguide,  Dirichlet Laplacian, integrated density of states, Lifshits tails

\section{Introduction}
\label{s0} \setcounter{equation}{0}

The spectral properties of quantum Hamiltonians on tubular domains (waveguides) have been actively studied for several decades (see  the monograph \cite{ek_book}, the survey \cite{krej}, and the references cited there). Recently, there has been a particular interest in the so-called {\em twisted waveguides} (see \cite{ek, ekk, bkrs, kz, bkr, bhk, r}), whose general setting we are going to describe briefly below. \\
Let $\cs \subset \rd$ be a bounded domain. Set $\dom : = \cs \times \re$. Let $\theta \in C^1(\re ; \re)$ have a bounded  derivative $\dot{\theta}$.  Define the twisted tube
$$
\domth : = \left\{{\mathcal R}_\theta(x_3) x \; | \; x=(x_{1},x_{2},x_{3}) \in \dom\right\}
$$ where ${\mathcal R}_\theta(x_3) $ is a rotation around the $x_{3}$-axis by an angle $\theta(x_{3})$, namely:
    \bel{10}
    {\mathcal R}_\theta(x_3) : = \left(
    \begin{array} {ccc}
    \cos{\theta(x_3)} & \sin{\theta(x_3)} & 0\\
    -\sin{\theta(x_3)} & \cos{\theta(x_3)} & 0\\
    0 & 0 &1
    \end{array} \right), \quad x_3 \in \re.
    \ee

We consider the Dirichlet Laplacian ${\mathcal H}_\theta$ defined through the corresponding quadratic form and suppose that $\theta$ is a random function described below in detail.
We are interested in the spectral properties
of ${\mathcal H}_\theta$, in particular in the behavior of the integrated density of states $N(E)$ of ${\mathcal H}_\theta$ for energies $E$ close to the bottom of the spectrum.

There are two immediate observations. First, if the set $m$ is invariant under rotations around the origin then the twist has no influence on the set $\mathcal{M}_{\theta} $, thus
$\mathcal{M}_{\theta} = M$. So a necessary condition for an effect of the random twist is that $m$ is not rotationally symmetric. We measure the `deviation' of $m$ from spherical symmetry
through the ground state $\varphi_{1}$ of the Dirichlet Laplacian on $m$ by the quantity $\T : = \nft$. In fact, one can prove (see \cite[Proposition 2.2]{bkrs}) that
$\T=0$ if and only if $m$ is spherical symmetric. Thus, in the following we assume that $\T\not=0$.

The second observation is the fact that a constant `twist' $\theta(x)\equiv \theta_{0}$, which is actually just a constant rotation of the set $M$,  does not effect the spectral properties of
${\mathcal H}_\theta$. Consequently, if $\tilde{\theta}(x)=\theta(x)+c$ then ${\mathcal H}_{\tilde{\theta}}$ and
${\mathcal H}_\theta$ are unitarily equivalent and have the same integrated density of states. It is thus the derivative $\dot{\theta}$ which determines the spectral properties of
${\mathcal H}_\theta$.
In this paper we consider random twists $\theta$ with $\dot{\theta}$ of the form

\begin{align}
   \dot{\theta}(x_{3})~=~\sum_{k\in\mathbb{Z}}\,\lambda_{k}(\omega)\,w(x_{3}-k) \label{mfo1}
\end{align}

where $\lambda_{k}$ are independent identically distributed random variables with common distribution $P_{0}$.
We assume that the probability measure $P_{0}$ is not concentrated in a single point and that its support ${\rm supp}\,P_{0}$ is compact and contains the origin.

The single-site twisting $w\in C^{1}$ is supposed to decay at infinity fast enough, namely
\bel{w1}
    |w(s)| \leq C (1 + |s|)^{-\alpha}, \quad s \in \re,
    \ee
for some $\alpha>1$.

Under these assumptions we show that the spectrum $\Sigma=\sigma(\mathcal{H}_{\theta}) $ is (almost surely) non random and the bottom $\Sigma_{0}$ of $\Sigma$
is the ground state energy $\mu_{1}$ of the Dirichlet Laplacian on $m$.

Our main results concern the asymptotic behavior of the integrated density of states $N$ of $\mathcal{H}_{\theta}$ near $\Sigma_0$.
For disordered systems this behavior is usually characterized by a very fast decay of the integrated density of states,
and is known as a {\em Lifshits-tail behavior}. For an overview on this topic see e.\,g. 	\cite{kme}, and references given there. Lifshits tails  concerning various random 2D waveguides were considered in \cite{klst, najar}. Related spectral properties were studied in \cite{bv0, bv}.\\

We will show that under suitable assumptions
    \bel{m50}
    \lim_{E \downarrow 0} \frac{\ln{|\ln{N(\Sigma_0 + E)}|}}{\ln{E}} = -\varkappa
    \ee
    with a constant $\varkappa > 0$ called {\em the Lifshits exponent} which depends, as we will see, on the decay rate of the single-site twisting $w$.

From the above discussion we know already that the randomness in our model can have no effect on the spectral properties if the cross section $m$ is rotationally
invariant. Indeed, in this case $N$ does not decay exponentially near $\Sigma_{0}$. Instead there is a van Hove singularity
, i.e. a non smooth power-like decay, instead of a Lifshits tail (see e.g. \cite{cdv} and the references cited there for a general discussion of the van Hove singularities). Thus, to obtain Lifshits tails we have to assume that $\mathcal{T}\not=0$.

Moreover, for our method of proof, we also need the assumption that the diameter of $m$ is small enough (see \eqref{d4} and \eqref{d6}).

Then our main results may be summarized in the following somewhat informal manner:

\begin{theorem} \label{main}
Suppose that the cross-section $m$ is not rotationally symmetric, and that the waveguide is thin enough.
Assume moreover that the random twisting $\dot{\theta}$ is defined as in \eqref{mfo1}, and the single-site twisting $w$  obeys the decay condition \eqref{w1} with $\alpha > 1$.\\
  {\rm (i)} If $\alpha \geq 2$, then \eqref{m50} holds true with Lifshits exponent $\varkappa =\frac{1}{2}$.\\
  {\rm (ii)} If $w(s) \sim |s|^{-\alpha}$ as $|s| \to \infty$, with $\alpha \in (1,2)$, then $\varkappa = \frac{1}{2(\alpha-1)}$.
 \end{theorem}
 In the following sections we will describe, among other things, the  details of our assumptions on the randomly twisted waveguide. Having provided the reader with these technicalities,
 we state in Section 4 our Theorems \ref{ft1} (i), \ref{th1}, and \ref{fth1} which could be considered as the rigorous versions of the first part of Theorem \ref{main} dealing with rapidly decaying $w$. Similarly, Theorem \ref{ft1} (ii) is the precise version of Theorem \ref{main} (ii) concerning single site-twisting $w$ of a slow decay. \\

    Let us say a few more words about the organization of the article. In the next section we give precise definitions of fundamental quantities and prove some basic properties of our models. In section 3 we estimate $N(\Sigma_0 + E)$ with small $E>0$ in terms of the integrated density of states for suitable 1D Schr\"odinger operators $h_{\dot{\theta}, \eps}$ (see \eqref{f1} below) whose potential depends on the random twisting $\dot{\theta}$ and on the real parameter $\eps$. In Section 4, we formulate and prove our main results on the Lifshits tails for $N$, applying the estimates obtained in Section 3, as well as certain results on the Lifshits tails for the operator $h_{\dot{\theta}, \eps}$. Some of these necessary results turned out to be available in the literature (see
    \cite{klonak, shen}) and some of them are borrowed from our companion paper \cite{kr} where Lifshits tails for Schr\"odinger operators with squared Anderson-type potentials are investigated in any dimension $d \geq 1$.

\section{Definitions and preliminary results}
Let ${\mathcal H}_\theta$ be the self-adjoint operator generated in ${\rm L}^2(\domth)$ by the closed quadratic form
$$
{\mathcal Q}_\theta[u] : = \int_{\domth} |\nabla u|^2 dx, \quad u \in {\rm H}^1_0(\domth),
$$
where, as usual, ${\rm H}^1_0(\domth)$ is the closure of $C^\infty_0(\domth)$ in the first-order Sobolev space ${\rm H}^1(\domth)$.
Introduce the quadratic form
$$
Q_{\dth}[u] : = \int_{\dom} \left( |\nabla_t u|^2 + |\dth \partial_\tau u + \partial_3 u|^2\right) dx, \quad u \in {\rm H}^1_0(\dom),
$$
where $\nabla_t : = (\partial_1, \partial_2)$, and $\partial_\tau : = x_1 \partial_2 - x_2 \partial_1$. Let $H_{\dth}$ be the self-adjoint operator generated in ${\rm L}^2(\dom)$ by the closed quadratic form $Q_{\dth}$. Define the unitary operator $U_\theta : {\rm L}^2(\domth) \to {\rm L}^2(\dom)$ by
$$
\left(U_\theta u \right)(x) : = u\left({\mathcal R}_\theta(x_3)\,x\right), \quad x \in \dom, \quad u \in {\rm L}^2(\domth).
$$
Then $H_{\dth}  = U_\theta {\mathcal H}_\theta U_\theta^{-1}$.\\

  If $\cs \subset \rd$ is a bounded domain with  boundary $\partial \cs \in C^2$, and  $\theta \in C^2(\re ; \re)$ has bounded first and second derivatives, then
    \bel{12}
H_{\dth} = -\partial_1^2 - \partial_2^2  -(\dth \partial_\tau  + \partial_3)^2, \quad {\rm Dom} (H_{\dth})  = {\rm H}^2(\dom) \cap {\rm H}_0^1(\dom),
    \ee
    (see \cite[Corollary 2.2]{bkr}).

In this article we will consider the operator $H_\gamma$ with random Anderson-type twisting $\dth = \gamma$ (see \eqref{ujf3} below).
Let $(\Omega, {\mathcal F}, {\mathbb P})$ be a probability space. Assume that $\lambda_k(\omega)$, $k \in \ze$, $\omega \in \Omega$, are independent, identically distributed random variables.
Set
$$
\lambda^- : = \essinf_{\omega \in \Omega}\,\lambda_0(\omega), \quad \lambda^+ : = \esssup_{\omega \in \Omega}\,\lambda_0(\omega).
$$
Throughout the article we assume that
    \bel{ujf1}
    -\infty < \lambda^- <  \lambda^+ < \infty.
    \ee
    Further, introduce the {\em single-site twisting} $w \in C(\re; \re)$ which is supposed to satisfy
    \bel{ujf2}
    |w(s)| \leq C (1 + |s|)^{-\alpha}, \quad s \in \re,
    \ee
    with some constants $C \in (0,\infty)$, and $\alpha \in (1,\infty)$. Moreover, we assume that
    \bel{f2}
    w \not \equiv 0 \quad \mbox{on} \quad \re.
    \ee

    Introduce the random twisting
    \bel{ujf3}
\gamma(s; \omega)~=~\dot{\theta}(s,\omega)  = \sum_{k \in \ze}\lambda_k(\omega) w(s-k), \quad s \in \re, \quad \omega \in \Omega.
    \ee
    Then $\gamma$ is a $\ze$-ergodic random field, and the operator $H_\gamma$, self-adjoint in ${\rm L}^2(\dom)$, is ergodic with respect to the translations $T_k$,  defined by
    $$
    (T_k u)(x_t,x_3) = u(x_t,x_3-k), \quad k \in \ze, \quad (x_t, x_3) \in \dom, \quad u \in {\rm L}^2(\dom).
    $$
 By the general theory of ergodic operators (see e.g. \cite[Section 4]{kirsch}), there exists a closed non-random subset $\Sigma$ of $\re$ such that almost surely
    \bel{4}
    \sigma(H_{\gamma}) = \Sigma.
    \ee
    Let us introduce {\em the integrated density of states} (IDS) of the operator $H_\gamma$.
    For a finite $\ell > 0$, set $\doml: = \cs \times (-\ell/2,\ell/2)$, and define the operator $H_{\gamma,\ell}$ as the self-adjoint operator generated in ${\rm L}^2(\doml)$ by the closed quadratic form
$$
Q_{\gamma,\ell}[u] = \int_{\doml} \left(|\nabla_t u|^2 + |\gamma(x_3; \omega) \partial_\tau u + \partial_3 u|^2\right)dx, \quad u \in {\rm H}_0^1(\doml).
$$
Evidently, the spectrum of $H_{\gamma,\ell}$ is purely discrete.
    We will say that the non-decreasing left-continuous function $N=N_\gamma: \re \to [0,\infty)$ is an IDS for the operator $H_\gamma$ if almost surely we have
    \bel{ujf21}
    \lim_{\ell \to \infty} \ell^{-1} \, {\rm Tr} \, \one_{(-\infty, E)}(H_{\gamma,\ell}) = N_\gamma(E)
    \ee
    at the points of continuity $E \in \re$ of $N_\gamma$.
    Arguing as in \cite[Theorem 6, Section 7]{kirsch} or \cite{hlmw1}, it is easy to show that there exists an IDS $N_\gamma$ for $H_\gamma$, and ${\rm supp}\,dN_\gamma = \Sigma$ (see \eqref{4}).\\

 \section{Estimates of $N_\gamma$ in terms of the IDS for 1D random Schr\"odinger operators}
\label{s4}

In this section we show that if $\essinf_{\omega \in \Omega} \lambda_0(\omega)^2 = 0$, then almost surely $\inf \, \sigma(H_\gamma)$ coincides with  $\mu_1$, the lowest eigenvalue of the transversal Dirichlet Laplacian, and obtain suitable two-sided estimates of $N(\mu_1 + E)$ for sufficiently small $E>0$, in terms of the IDS for appropriate 1D random Schr\"odinger operators $h_{\gamma, \eps}$ (see \eqref{f1} below). \\

Let $\left\{\mu_j\right\}_{j \in \N}$ be the non-decreasing sequence of the eigenvalues of the transversal Dirichlet Laplacian $-\Delta_t^D$, generated in ${\rm L}^2(m)$ by the closed quadratic form
$$
\int_{\cs} |\nabla_t u|^2 dx_t, \quad u \in {\rm H}^1_0(\cs),
$$
with $x_t : = (x_1,x_2)$. We have
    \bel{ujf4}
    0 < \mu_1 < \mu_2.
    \ee
Let $\left\{\varphi_j\right\}_{j \in \N}$  be an orthonormal basis in ${\rm L^2}(\cs)$ consisting of real-valued eigenfunctions of $-\Delta_t^D$ which satisfy
$$
-\Delta_t^D\, \varphi_j = \mu_j \varphi_j, \quad j \in \N.
$$
It is well known that $\varphi_1$ could be chosen so that
$$
\varphi_1(x_t) > 0, \quad x_t  \in \cs.
$$
Set
    \bel{d1}
    \T : = \nft.
    \ee

%
%
%
%

    Arguing as in the proof of \cite[Proposition 2.2]{bkrs}, we can show that if $\partial \cs \in C^2$, then the inequality
    \bel{m1}
    \T \neq 0
    \ee
holds true if and only if $\cs$ is not rotationally symmetric with respect to the origin.
On the other hand, if  $\cs$ is any bounded rotationally symmetric domain, then $\T = 0$. Moreover, in this case
the operator $H_{\dth}$ is unitarily equivalent to $H_0$, the spectrum $\sigma(H_{\dth}) = [\mu_1, \infty)$ is absolutely continuous, the IDS $N_{\dth} = N_0$, independent of $\dth$, is well defined by analogy with \eqref{ujf21},  and we have
    \bel{m51}
N_0(E) = \frac{1}{\pi} \sum_{j=1}^\infty (E-\mu_j)_+^{1/2}, \quad E \in \re.
    \ee
In particular,
    \bel{m2}
    N_0(\mu_1 + E) = \frac{1}{\pi} E_+^{1/2}, \quad E \in (-\infty, \mu_2-\mu_1).
    \ee

Assume \eqref{ujf1}, \eqref{ujf2}, and
    \bel{d3}
    w \in C^1(\re;\re), \quad |\dot{w}(s)| \leq C (1+ |s|)^{-\alpha}, \quad s \in \re.
    \ee
For $\eps \in \re$ introduce the operator $h_{\gamma, \eps}$ as the self-adjoint operator generated in ${\rm L}^2(\re)$
by the closed quadratic form
$$
q_{\gamma, \eps}[f] : = \int_\re \left(|\dot{f}|^2 + \left(\Td \gamma(s;\omega)^2 - \eps \dot{\gamma}(s;\omega)^2\right)|f|^2\right) ds, \quad f \in {\rm H}^1(\re).
$$
{\em Remark}: If $\eps = 0$, then we can omit assumption \eqref{d3} in the definition of the operator $h_{\gamma, \eps}$. \\

Thus,
    \bel{f1}
h_{\gamma, \eps}  = - \frac{d^2}{ds^2} + \Td \gamma^2 - \eps \dot{\gamma}^2
    \ee
is a 1D Schr\"odinger operator with random potential $\Td \gamma(s;\omega)^2 - \eps \dot{\gamma}(s;\omega)^2$, $s \in \re$, $\omega \in \Omega$. This operator is $\ze$-ergodic, and its spectrum is almost surely independent of $\omega \in \Omega$.
Introduce the IDS for the operator $h_{\gamma, \eps}$ as the non-decreasing function $\nu_{\gamma, \eps} : \re \to \re$ which almost surely satisfies
    \bel{ujf20}
\lim_{\ell \to \infty} \ell^{-1} \, {\rm Tr} \, \one_{(-\infty, E)}(h_{\gamma,\eps, \ell}) = \nu_{\gamma, \eps}(E), \quad E \in \re,
    \ee
    $h_{\gamma,\eps, \ell}$ being the self-adjoint operator generated in ${\rm L}^2(-\ell/2,\ell/2)$ by the closed quadratic form
    \bel{d2}
q_{\gamma, \eps, \ell}[f] : = \int_{-\ell/2}^{\ell/2} \left(|\dot{f}|^2 + \left(\Td \gamma(s;\omega)^2 - \eps \dot{\gamma}(s;\omega)^2\right)|f|^2\right) ds, \quad f \in {\rm H}^1_0(-\ell/2, \ell/2).
    \ee
The IDS $\nu_{\gamma, \eps}$ exists and is continuous (see \cite[Theorem 3.2]{pf}). Moreover, in the definition \eqref{ujf20} of $\nu_{\gamma, \eps}$, we can replace the operator $h_{\gamma,\eps, \ell}$ equipped with Dirichlet boundary conditions by the operator generated by the quadratic form \eqref{d2} with domain
${\rm H}^1(-\ell/2, \ell/2)$,  corresponding to Neumann boundary conditions.
 Further, it follows from \eqref{ujf1} that
$$
\tilde{\lambda}^+ : = \esssup_{\omega \in \Omega}  \lambda_0(\omega)^2  > 0.
$$
In what follows, we assume that
    \bel{ujf8}
    \tilde{\lambda}^-: = \essinf_{\omega \in \Omega} \lambda_0(\omega)^2 = 0.
    \ee
    Note that \eqref{ujf8} implies that almost surely
    \bel{ujf8a}
      \sigma(h_{\gamma, 0}) = [0,\infty)
    \ee
    (see \cite{kma}).
	
    \begin{pr} \label{p1}
    Assume \eqref{ujf1}, \eqref{ujf2}, and \eqref{ujf8}.
    Then almost surely we have
    \bel{ujf9}
    \sigma(H_{\gamma}) = [\mu_1,\infty).
    \ee
    \end{pr}
    \begin{proof}
    We have
    \bel{m60}
    \inf\,\sigma(H_{\gamma}) =
   \inf_{0 \neq u \in {\rm H}^1_0(M)} \frac{Q_{\gamma}[u]}{\| u \|^2_{{\rm L}^2(M)}}.
    \ee
    Since
    $$
    Q_{\gamma}[u] \geq \int_M |\nabla_t u|^2 \,dx, \quad u \in {\rm H}_0^1(\dom),
    $$
   it follows from \eqref{m60} and
   $$
   \mu_1 = \inf_{0 \neq u \in {\rm H}^1_0(M)} \frac{\int_M |\nabla_t u|^2 \,dx}{\int_M |u|^2\,dx},
   $$
   that
    \bel{ujf10}
    \inf \sigma(H_{\gamma}) \geq \mu_1.
    \ee

Let us now prove the almost sure inclusion
\begin{equation}\label{opposite}
  \sigma(H_{\gamma}) \supset [\mu_1,\infty) .
\end{equation}
Fix $E \geq 0$. Arguing along the lines of the proof of \eqref{ujf8a} in \cite{kma}, we can construct a
sequence $\{f_n\}_{n\in\N} \subset C_0^\infty(\R)$,
normalized to one in $\mathrm{L}^2(\R)$,
such that, almost surely
\begin{equation}\label{Weyl1}
  \|-\ddot{f}_n- E f_n\|_{\mathrm{L}^2(\R)} \xrightarrow[n\to\infty]{} 0
  \qquad \mbox{and} \qquad
  \|\gamma\|_{\mathrm{L}^\infty(\supp f_n)} \xrightarrow[n\to\infty]{} 0
  .
\end{equation}
Notice that, by writing
$
  \|\dot{f}_n\|_{\mathrm{L}^2(\R)}^2
  = -(\ddot{f}_n,f_n)_{\mathrm{L}^2(\R)}
  \leq \|\ddot{f}_n\|_{\mathrm{L}^2(\R)}
$,
it follows from the first limit in~\eqref{Weyl1}
that the sequence $\{\dot{f}_n\}_{n\in\N}$
is almost surely bounded in $\mathrm{L}^2(\R)$.
The sequence $\{u_n\}_{n\in\N} \subset \mathrm{H}^1(M)$
defined by
$$
    u_n := \varphi_1 \otimes f_n
$$
is normalized to one in $\mathrm{L}^2(M)$.
By the Weyl criterion adapted to quadratic forms (see~\cite[Theorem~5]{KL}),
the desired inclusion~\eqref{opposite}
will hold  if we show that, almost surely,
\begin{equation}\label{Weyl2}
  \sup_{0 \neq \phi \in {\rm H}^1_0(M)}
  \frac{|Q_\gamma(u_n,\phi)-(\mu_1 + E) (u_n,\phi)_{\mathrm{L}^2(M)}|}
  {\|\phi\|_{\mathrm{H}^1(M)}}
  \xrightarrow[n\to\infty]{} 0 ,
\end{equation}
where $Q_\gamma(\cdot,\cdot)$ is the sesquilinear form generated by the quadratic form $Q_\gamma[u]$, $u \in {\rm H}_0^1(M)$, and $(\cdot ,\cdot)_{{\rm L}^2(M)}$ is the scalar product in ${\rm L}^2(M)$. \\
Integrating by parts, using the normalizations of~$f_n$ and~$\varphi_1$,
and applying the Cauchy-Schwarz inequality,
we get
\begin{align} \label{may}
  |Q_\gamma(u_n,\phi)-(\mu_1 + E) (u_n,\phi)_{\mathrm{L}^2(M)}|
  \ \leq \ & \|\phi\|_{\mathrm{L}^2(M)} \, \|-\ddot{f}_n - E f_n\|_{\mathrm{L}^2(\R)}
  \\
  & + \|\partial_3\phi\|_{\mathrm{L}^2(M)} \,
  \|\gamma\|_{\mathrm{L}^\infty(\supp f_n)} \,
  	\mathcal{T} \nonumber
  \\
  & + \|\partial_\tau\phi\|_{\mathrm{L}^2(M)} \,
  \|\gamma\|_{\mathrm{L}^\infty(\supp f_n)} \,
  \|\dot{f}_n\|_{\mathrm{L}^2(\re)} \nonumber
  \\
  & + \|\partial_\tau\phi\|_{\mathrm{L}^2(M)} \,
  \|\gamma^2\|_{\mathrm{L}^\infty(\supp f_n)} \,
  \mathcal{T} \nonumber
  .
\end{align}
Thus, \eqref{may} and \eqref{Weyl1} imply \eqref{Weyl2}, and hence \eqref{opposite}.\\
 Now \eqref{ujf9} follows from \eqref{ujf10} and \eqref{opposite}.
\end{proof}

  Further, we need several notations which will allow us to formulate certain assumptions of geometric nature.
  Assume \eqref{ujf1}, \eqref{ujf2}, and set
  $$
  D_1 : =  \esssup_{\omega \in \Omega} \sup_{s \in \re} (5\gamma(s;\omega)^2 +1).
  $$
  Then $D_1 < \infty$.\\
   Further,
  assume \eqref{ujf1}, \eqref{ujf2}, \eqref{d3}, and \eqref{m1}. Suppose in addition that the logarithmic derivative $\dot{\gamma}/\gamma$ is well defined
  and
  \bel{m63}
  \esssup_{\omega \in \Omega} \sup_{s \in \re} \left| \frac{\dot{\gamma}(s;\omega)}{\gamma(s,\omega)}\right| < \infty.
  \ee
   Set
  $$
  D_2 : =  \esssup_{\omega \in \Omega} \sup_{s \in \re} \left(6\gamma(s;\omega)^2 + \frac{2\dot{\gamma}(s;\omega)^2}{\Td \gamma(s;\omega)^2}\right).
  $$
  Then $D_2 < \infty$. \\

  {\em Remark}: Assumption \eqref{m63} holds true if $w$ does not vanish at any $s \in \re$ and admits a regular power-like decay at infinity, but it is false if $w$ has a compact support.\\

  Finally, put
  $$
  a: = \sup_{x_t \in \cs} |x_t|.
  $$

  \begin{theorem} \label{dt1}
  Assume \eqref{ujf1} and \eqref{ujf2}. \\
  {\rm (i)} We have
  \bel{d8}
  \nu_{\gamma, 0}(E) \leq N_\gamma(\mu_1 + E), \quad E \in \re.
  \ee
  {\rm (ii)} Let $\delta_0 \in (0,1)$. Suppose in addition that \eqref{d3} holds true, and
  \bel{d4}
  a^2 \left(1 - \frac{\mu_1}{\mu_2}\right)^{-1}  D_1  < \delta_0.
  \ee
   Then we have
   \bel{d5}
    N_\gamma(\mu_1 + E) \leq \nu_{\gamma, \delta/(1-\delta)}((1-\delta)^{-1}E)
   \ee
   for any $\delta \in \left(a^2 \left(1 - \frac{\mu_1}{\mu_2}\right)^{-1}  D_1, \delta_0\right)$ and $E \in (0,\mu_2(1-\delta^{-1} D_1 a^2)-\mu_1)$.\\
   {\rm (iii)} Suppose in addition that \eqref{d3}, \eqref{m1}, and \eqref{m63} hold true, and
  \bel{d6}
  a^2 \left(1 - \frac{\mu_1}{\mu_2}\right)^{-1} D_2  < 1.
  \ee
   Then we have
   \bel{d7}
    N_\gamma(\mu_1 + E) \leq \nu_{\gamma, 0}((1-\delta)^{-1}E)
   \ee
   for any $\delta \in \left(a^2 \left(1 - \frac{\mu_1}{\mu_2}\right)^{-1}  D_2, 1\right)$ and $E \in (0,\mu_2(1-\delta^{-1} D_2 a^2)-\mu_1)$.
  \end{theorem}
  {\em Remark}: If $\gamma$ is fixed and $D_1 < \infty$ (resp., $D_2 < \infty$), then \eqref{d4}  (resp., \eqref{d6}) holds true if $a$ is small enough. Note that it follows from the results of \cite{deo, ks} that the operator $H_\gamma - \mu_1$ converges in an appropriate sense to $h_{\gamma,0}$ as $a \downarrow 0$ .\\
  \begin{proof}[Proof of Theorem \ref{dt1}]
   If we restrict the quadratic form $Q_{\gamma, \ell}$ to functions of the form
    $$
    u_1 = \varphi_1 \otimes f, \quad f \in {\rm H}_0^1(-\ell/2,\ell/2),
    $$
    then
    \bel{m61}
    Q_{\gamma, \ell}[u_1] = q_{\gamma,0,\ell}[f] + \mu_1 \|f\|^2_{{\rm L}^2(-\ell/2,\ell/2)}, \quad \|u_1\|^2_{{\rm L}^2(\dom)} = \|f\|^2_{{\rm L}^2(\re)},
    \ee
    the quadratic form $q_{\gamma,\eps,\ell}$ being defined in \eqref{d2}.
    Hence, the mini-max principle implies
    \bel{ujf19}
    {\rm Tr}\,\one_{(-\infty,\mu_1 + E)}(H_{\gamma, \ell}) \geq {\rm Tr}\,\one_{(-\infty,E)}(h_{\gamma, 0, \ell}), \quad E \in \re.
    \ee
    Combining \eqref{ujf21}, \eqref{ujf20}, and \eqref{ujf19}, we get \eqref{d8}. \\
    Next, set
    $$
    {\mathcal D}_1 : = \left\{u_1 = \varphi_1 \otimes f \, | \, f \in {\rm H}^1_0(-\ell/2,\ell/2)\right\},
    $$
    $$
    {\mathcal D}_2 : = \left\{u_2 \in {\rm H}_0^1(\doml) \, | \, \int_{\doml} u_2(x) \overline{u_1(x)} dx = 0, \, \forall u_1 \in {\mathcal D}_1\right\}.
    $$
     Then, for $u = u_1 + u_2$ with $u_1 = \varphi_1 \otimes f \in {\mathcal D}_1$ and $u_2 \in {\mathcal D}_2$, we have
    $$
    \|u\|_{{\rm L}^2(\doml)}^2 = \|u_1 + u_2\|_{{\rm L}^2(\doml)}^2 = \|f\|^2_{{\rm L}^2(-\ell/2,\ell/2)} + \|u_2\|_{{\rm L}^2(\doml)}^2.
    $$
    Moreover, integrating by parts, we get
    $$
    Q_{\gamma, \ell}[u] = Q_{\gamma,\ell}[u_1+u_2] =
    $$
    $$
    Q_{\gamma,\ell}[u_1] +  Q_{\gamma,\ell}[u_2] + 2{\rm Re}\,\int_{\doml} \left(\gamma^2 \partial_\tau u_1 \overline{\partial_\tau u_2} +  \gamma \partial_3 u_1 \overline{\partial_\tau u_2}
     + \gamma \partial_\tau u_1 \overline{\partial_3 u_2}\right) dx =
     $$
    \bel{d9}
    Q_{\gamma,\ell}[u_1] +  Q_{\gamma,\ell}[u_2] + 2{\rm Re}\,\int_{\doml} \left(\gamma^2 \partial_\tau u_1  + 2 \gamma \partial_3 u_1
    + \dot{\gamma} u_1 \right)\, \overline{\partial_\tau u_2}\, dx.
     \ee
     Assume \eqref{d4} and pick $\delta \in  \left(a^2 \left(1 - \frac{\mu_1}{\mu_2}\right)^{-1}  D_1, \delta_0\right)$. We have
     $$
     2{\rm Re}\,\int_{\doml} \left(\gamma^2 \partial_\tau u_1  + 2 \gamma \partial_3 u_1
    + \dot{\gamma} u_1 \right)\, \overline{\partial_\tau u_2}\, dx \geq
     $$
     $$
     -\delta \int_{M_\ell} \left(\gamma^2 |\partial_\tau u_1|^2 + |\partial_3 u_1|^2 +  \dot{\gamma}^2 |u_1|^2\right)dx
     -\delta^{-1} \int_{M_\ell} (5 \gamma^2 + 1) |\partial_\tau u_2|^2 dx =
     $$
     \bel{a1}
     -\delta \int_{-\ell/2}^{\ell/2} \left(|\dot{f}|^2 + ({\mathcal T}^2 \gamma^2 + \dot{\gamma}^2)|f|^2\right)dx_3
     -\delta^{-1} \int_{M_\ell} (5 \gamma^2 + 1) |\partial_\tau u_2|^2 dx.
     \ee
     Then, \eqref{m61}, \eqref{d9}, and \eqref{a1}  easily imply
      \bel{d10}
      Q_{\gamma,\ell}[u] \geq
     (1-\delta) q_{\gamma, \delta/(1-\delta), \ell} [f] + \mu_1 \|f\|^2_{{\rm L}^2(-\ell/2,\ell/2)} + \tilde{Q}_{\gamma,\ell}[u_2]
     \ee
     where
     $$
     \tilde{Q}_{\gamma,\ell}[u_2] : =
     $$
     $$
     \int_{\doml} \left(|\nabla_t u_2|^2 - \delta^{-1} (5 \gamma^2 +1)|\partial_\tau u_2|^2 + |\gamma \partial_\tau u_2 + \partial_3 u_2|^2 \right) dx, \quad u_2 \in {\mathcal D}_2.
     $$
     Let $\tilde{H}_{\gamma, \ell}$ be the operator generated by  $ \tilde{Q}_{\gamma,\ell}$ in the Hilbert space ${\mathcal D}_1^\perp$, the orthogonal complement of  ${\mathcal D}_1$ in ${\rm L}^2(\doml)$. Then the mini-max principle and \eqref{d10} imply
     \bel{d20}
     {\rm Tr}\,\one_{(-\infty,\mu_1 + E)}(H_{\gamma, \ell}) \leq {\rm Tr}\,\one_{(-\infty,E)}((1-\delta) h_{\gamma, \delta/(1-\delta), \ell}) +
     {\rm Tr}\,\one_{(-\infty,\mu_1 + E)}(\tilde{H}_{\gamma, \ell}), \quad E \in \re.
     \ee
     Since $|\partial_\tau u_2| \leq |x_t| |\nabla_t u_2|$,
     we have
     \bel{d14}
     \tilde{Q}_{\gamma,\ell}[u_2] \geq \mu_2 \left( 1- \delta^{-1} a^2 D_1\right) \int_{\doml} |u_2|^2 dx.
     \ee
      Therefore, if $E \in (0,\mu_2 \left( 1- \delta^{-1} a^2 D_1\right) - \mu_1)$, we have
      $$
      {\rm Tr}\,\one_{(-\infty,\mu_1 + E)}(\tilde{H}_{\gamma, \ell}) = 0,
      $$
      and by \eqref{d20},
     \bel{d11}
    {\rm Tr}\,\one_{(-\infty,\mu_1 + E)}(H_{\gamma, \ell}) \leq {\rm Tr}\,\one_{(-\infty,E)}((1-\delta) h_{\gamma, \delta/(1-\delta), \ell}) ={\rm Tr}\,\one_{(-\infty,(1-\delta)^{-1}E)}( h_{\gamma, \delta/(1-\delta), \ell}).
    \ee
      Now \eqref{ujf21}, \eqref{ujf20}, and \eqref{d11}, imply \eqref{d5}.\\
      Finally, assume \eqref{d6} and pick $\delta \in  \left(a^2 \left(1 - \frac{\mu_1}{\mu_2}\right)^{-1}  D_2, 1\right)$. Similarly to \eqref{a1}, we have
      $$
     2{\rm Re}\,\int_{\doml} \left(\gamma^2 \partial_\tau u_1  + 2 \gamma \partial_3 u_1
    + \dot{\gamma} u_1 \right)\, \overline{\partial_\tau u_2}\,  dx \geq
     $$
     $$
     -\delta \int_{M_\ell} \left(\frac{\gamma^2}{2} |\partial_\tau u_1|^2 + |\partial_3 u_1|^2 +  \frac{\Td \gamma^2}{2} |u_1|^2\right)dx
     -\delta^{-1} \int_{M_\ell} \left(6 \gamma^2 + \frac{2\dot{\gamma}^2}{\Td \gamma^2}\right) |\partial_\tau u_2|^2 dx =
     $$
     $$
     -\delta \int_{-\ell/2}^{\ell/2} \left(|\dot{f}|^2 + {\mathcal T}^2 \gamma^2|f|^2\right)dx_3
     -\delta^{-1} \int_{M_\ell} \left(6 \gamma^2 + \frac{2\dot{\gamma}^2}{\Td \gamma^2}\right) |\partial_\tau u_2|^2 dx.
     $$
      Hence, by analogy with \eqref{d10} and \eqref{d14},  we have
      $$
      Q_{\gamma, \ell}[u] \geq
      $$
      $$
     (1-\delta) q_{\gamma, 0, \ell} [f] + \mu_1 \|f\|^2_{{\rm L}^2(-\ell/2,\ell/2)} +
     $$
     $$
     \int_{\doml} \left(|\nabla_t u_2|^2 - \delta^{-1} \left(6 \gamma^2 + \frac{2\dot{\gamma}^2}{\Td \gamma^2}\right)|\partial_\tau u_2|^2 + |\gamma \partial_\tau u_2 + \partial_3 u_2|^2\right) dx \geq
     $$
     $$
      (1-\delta) q_{\gamma, 0, \ell} [f] + \mu_1 \|f\|^2_{{\rm L}^2(-\ell/2,\ell/2)} + \mu_2 \left( 1- \delta^{-1} a^2 D_2\right) \int_{\doml} |u_2|^2 dx.
     $$
      Therefore, if $E \in (0,\mu_2 \left( 1- \delta^{-1} a^2 D_2\right) - \mu_1)$, we have
     \bel{d13}
    {\rm Tr}\,\one_{(-\infty,\mu_1 + E)}(H_{\gamma, \ell}) \leq {\rm Tr}\,\one_{(-\infty,(1-\delta)^{-1}E)}( h_{\gamma, 0, \ell}).
    \ee
      Now \eqref{ujf21}, \eqref{ujf20}, and \eqref{d13}, imply \eqref{d7}.

  \end{proof}

  \section{Lifshits tails for the operator $H_\gamma$}
\label{s5}
In this section we formulate and prove our main results concerning the asymptotic behavior of $N_\gamma(\mu_1 + E)$ as $E \downarrow 0$. In Subsection \ref{ss51} we consider single-site twisting $w$ of power-like decay while in Subsection \ref{ss52} we handle the case of compactly supported $w$.

\subsection{Single-site twisting $w$ of power-like decay}
\label{ss51}

The following proposition contains results from \cite{kr} on the Lifshits tails for 1D Schr\"odinger operators with squared random Anderson-type potentials.
\begin{pr}[{\cite[Theorem 1]{kr}}] \label{fp1} 
Assume \eqref{m1}. Suppose that $w$ satisfies \eqref{ujf2} with $\alpha \in (1,\infty)$, and \eqref{f2}, while $\lambda_0$ satisfies
\eqref{ujf1} and \eqref{ujf8}.  Suppose moreover that
  \bel{ujf14}
  {\mathbb P}(\{\omega \in \Omega \, | \, |\lambda_0(\omega)| < \varepsilon\}) \geq C\varepsilon^\kappa,
  \ee
  for some $\kappa > 0$, $C >0$, and any sufficiently small $\varepsilon>0$.\\
{\rm (i)} If $\alpha \geq 2$, then
    \bel{f13}
\lim_{E \downarrow 0} \frac{\ln{|\ln{\nu_{\gamma,0}(E)}|}}{\ln{E}} = - \frac{1}{2}.
    \ee
{\rm (ii)} Let $1 < \alpha < 2$. Assume that
    \bel{f3}
    w(s) \geq C(1 + |s|)^{-\alpha}, \quad s \in \re, \quad C >0,
    \ee
    and
    \bel{f4}
    \lambda^- = 0.
    \ee
    Then
$$
\lim_{E \downarrow 0} \frac{\ln{|\ln{\nu_{\gamma,0}(E)}|}}{\ln{E}} = - \frac{1}{2(\alpha - 1)}.
$$
\end{pr}

{\em Remark}: Evidently, we may replace the assumptions \eqref{f3} and \eqref{f4}, by $w(s) \leq -C(1 + |s|)^{-\alpha}$, $s \in \re$, with $C >0$, and $\lambda^+  = 0$ respectively. A similar remark applies to Theorems \ref{ft1} (ii) and \ref{fth1}. \\

Combining Theorem \ref{dt1} with Proposition \ref{fp1}, we obtain the following theorem concerning the Lifshits tails of the IDS $N_\gamma$ for the randomly twisted waveguide:

\begin{theorem} \label{ft1} Let $\cs \subset \rd$ be a bounded domain such that ${\mathcal T} \neq 0$. Assume that:
\begin{itemize}
 \item $w \in C^1(\re; \re)$ does not vanish identically on $\re$ and satisfies  the upper bound \eqref{ujf2} with $\alpha \in (1,\infty)$;
     \item $\lambda_0$ satisfies
\eqref{ujf1},  \eqref{ujf8}, and \eqref{ujf14};
    \item the logarithmic derivative $\dot{\gamma}/\gamma$ satisfies the boundedness condition \eqref{m63};
     \item the waveguide satisfies ``the thinness condition" \eqref{d6}.
     \end{itemize}
     {\rm (i)} Let $\alpha \in [2,\infty)$. Then we have
    \bel{ujf17}
\lim_{E \downarrow 0} \frac{\ln{|\ln{N_{\gamma}(\mu_1 + E)}|}}{\ln{E}} = - \frac{1}{2}.
    \ee
  {\rm (ii)} Let $\alpha \in (1,2)$.
    Suppose moreover that the lower bounds \eqref{f3} and \eqref{f4} hold true. Then we have
    $$
\lim_{E \downarrow 0} \frac{\ln{|\ln{N_{\gamma}(\mu_1 + E)}|}}{\ln{E}} = - \frac{1}{2(\alpha-1)}.
    $$
  \end{theorem}
{\em Remark}: If ${\mathcal T} = 0$, then
    $$
    \nu_{\gamma, 0}(E) = \nu_{0, 0}(E) = \frac{1}{\pi} E_+^{1/2}, \quad E \in \re.
    $$
    Therefore, \eqref{d8}  implies
    $$
    \liminf_{E \downarrow 0} \frac{\ln{|\ln{N_\gamma(\mu_1 + E)}|}}{\ln{E}} \geq 0,
    $$
    i.e. $N_\gamma$ does not exhibit a Lifshits tail near $\mu_1$. As mentioned in the introduction, if $\partial \cs \in C^2$, then ${\mathcal T} = 0$ is equivalent to the fact that $\cs$ is rotationally invariant with respect to the origin, and \eqref{m51} and \eqref{m2} hold true, i.e. $N_\gamma$ exhibits near $\mu_1$ a van Hove singularity instead of a Lifshits tail. A similar remark applies to Theorems \ref{th1} and \ref{fth1}. \\

\subsection{Single-site twisting $w$ of compact support}
\label{ss52}
In this subsection we assume that \eqref{f2} holds true, and
    \bel{ujf6}
    w \in C^1(\re;\re), \quad {\rm supp}\,w \subset [-\beta/2,\beta/2],
    \ee
    with $\beta \in (0,\infty)$.\\
    First, we consider the case where the support of $w$ is small, i.e. \eqref{ujf6} holds with $\beta \in (0,1]$.
    Then the multiplier by $\Td \gamma(s;\omega)^2 - \eps \dot{\gamma}(s;\omega)^2$ coincides with the multiplier by
$$
 \sum_{k \in \ze} \lambda_k(\omega)^2 v_\eps(s-k), \quad s \in \re,
$$
where
    \bel{m5}
v_\eps(s) : = \Td w(s)^2 - \eps \dot{w}(s)^2, \quad s \in \re.
    \ee

  For $\eps \in \re$ denote by ${\mathcal E}^\pm(\eps)$ the lowest eigenvalue of the operator
  \bel{f10}
 h^\pm_\eps : =  -\frac{d^2}{ds^2} + \tilde{\lambda}^\pm v_\eps,
  \ee
 acting in ${\rm L}^2(-1/2,1/2)$, and equipped with Neumann boundary conditions.
 If \eqref{ujf8} is fulfilled, then, evidently,
 \bel{m64}
 {\mathcal E}^-(\eps) = 0, \quad \eps \in \re.
 \ee
 Put
  \bel{ujf23}
  \eps_0 : = \sup\,{\{\eps \in \re \, | \, {\mathcal E}^+(\eps) > 0\}}.
  \ee
  It follows from \eqref{f2} and \eqref{ujf1} that if \eqref{m1} is valid, then $\eps_0 > 0$ since ${\mathcal E}^+(0) > 0$, and ${\mathcal E}^+$ is a continuous (as a matter of fact, real analytic) non-increasing function of $\eps \in \re$.
  Thus,
  \bel{m65}
 {\mathcal E}^+(\eps) > 0, \quad \eps \in (-\infty, \eps_0).
 \ee

  \begin{pr} \label{p2}
  Assume that \eqref{m1} holds true, $w$ satisfies \eqref{f2}, \eqref{ujf6} with $\beta \in (0,1]$, while $\lambda_0$ satisfies \eqref{ujf1} and \eqref{ujf8}.  Let $\eps \in (-\infty,\eps_0)$.\\
  {\rm (i)} We have almost surely
  \bel{ujf15}
  \inf \sigma(h_{\gamma,\eps}) = 0.
  \ee
  {\rm (ii)} Moreover,
  \bel{ujf16}
  \limsup_{E \downarrow 0} \frac{\ln{|\ln{\nu_{\gamma,\eps}(E)}|}}{\ln{E}} \leq -\frac{1}{2}.
  \ee
  \end{pr}
 \begin{proof}[Idea of the proof of Proposition \ref{p2}:]
 Taking into account \eqref{ujf8}, \eqref{m64}, and \eqref{m65}, we find that \eqref{ujf15} follows from \cite[Proposition 0.1]{klonak}. Note that the hypotheses of \cite[Proposition 0.1]{klonak} contain also the condition that $v_\eps$ be an even function of $s \in \re$. However, this condition is needed to guarantee that the eigenfunction of the operator $h^-_\eps$ is even, which in our setting is immediately implied by \eqref{m64}.\\
 Further, bearing in mind \eqref{ujf15}, \eqref{m64}, and \eqref{m65}, we easily conclude that \eqref{ujf16} follows from \cite[Theorem 0.1]{klonak}. \\
 It should be noted here that the assumptions of Proposition 0.1 and Theorem 0.1 of \cite{klonak} require that ${\rm supp}\,v_\eps \subset (-1/2,1/2)$ which may formally exclude the case $\beta = 1$ in \eqref{ujf6}. A careful analysis of the proofs of Proposition 0.1 and Theorem 0.1 of \cite{klonak} however shows that these proofs extend without any problem to the case ${\rm supp}\,v_\eps \subset [-1/2,1/2]$. \\
 \end{proof}
{\em Remarks}: (i) Proposition \ref{p2} also follows from the results of the article \cite{shen} which extends \cite{klonak}. More precisely, \eqref{ujf15} follows from \cite[Theorem 1.1]{shen}, while \eqref{ujf16} follows from
\cite[Theorem 1.2]{shen}.\\
(ii) If $\eps \leq 0$ and hence $v_{\eps}$ does not change sign, \eqref{ujf15} and \eqref{ujf16} have been known since long ago (see \cite{kma} and \cite{ksim1} respectively). However, the case $\eps \leq 0$ is not appropriate for our purposes.

  \begin{theorem} \label{th1}
  Let $\cs \subset \rd$ be a bounded domain such that ${\mathcal T} \neq 0$.
  Assume that:
   \begin{itemize}
   \item $w$ does not vanish identically on $\re$ and satisfies  \eqref{ujf6} with $\beta \in (0,1]$;
   \item $\lambda_0$ satisfies \eqref{ujf1}, \eqref{ujf8}, and \eqref{ujf14};  \item the waveguide satisfies ``the thinness condition"  \eqref{d4}  with $\delta_0 = \frac{\eps_0}{1+\eps_0}$, $\eps_0$ being defined in \eqref{ujf23}.
       \end{itemize}
       Then
  \eqref{ujf17} is valid again.
  \end{theorem}
  \begin{proof}
 If $\delta < \frac{\eps_0}{1+\eps_0}$, then  $\delta/(1-\delta) < \eps_0$. Therefore, \eqref{ujf17} follows from \eqref{d8}, \eqref{d5}, \eqref{f13} and \eqref{ujf16}.
  \end{proof}

Further, we consider the case where the support of $w$ may be large, i.e. \eqref{f2}, and \eqref{ujf6} with $\beta \in (1,\infty)$ hold true; then the supports of the translates of $w$ may have a substantial overlap. Without any loss of generality, we assume that $\beta = 2p + 1$ with $p \in \N$. Set $\J : = \left\{-p,\ldots,p\right\}$, and
    $$
    \J_1 : = \left\{j \in \J \, | \, w \not \equiv 0 \quad \mbox{on} \left[-\frac{1}{2}+j, \frac{1}{2}+j\right]\right\},
$$
$$
    \J_2 : = \left\{j \in \J \, | \, \dot{w} \not \equiv 0 \quad \mbox{on} \left[-\frac{1}{2}+j, \frac{1}{2}+j\right]\right\},
$$
$$
n_k := \# \J_k, \quad k=1,2.
$$
Evidently, $\J_2 \subset \J_1$, and $n_1 \geq n_2 \geq 1$.
By analogy with \eqref{m5}, set
    \bel{f19}
v_{j,\eps}(s) := \left(\Td w(s+j)^2 - n_2 \eps \dot{w}(s+j)^2\right) \one_{[-1/2,1/2)}(s), \quad s \in \re, \quad \eps \in \re, \quad j \in \J,
    \ee
so that $\supp\,{v_{j,\eps}} \subset [-1/2,1/2]$.  By analogy with \eqref{f10}, for $\eps \in \re$, consider the Neumann realization of the operators
    \bel{f11}
 h^\pm_{j,\eps} : =   -\frac{d^2}{ds^2} + n_1 \tilde{\lambda}^\pm v_{j,\eps}, \quad j \in \J_1,
    \ee
    restricted on $(-1/2,1/2)$. Denote by ${\mathcal E}^\pm_j(\eps)$, $j \in \J_1$, the lowest eigenvalue of the operator $h^\pm_{j,\eps}$. Put
    $$
    \eps_0^{\rm min} : = \min_{j \in \J_1} \sup{\left\{\eps \in \re \, | \, {\mathcal E}^+_j(\eps)>0\right\}}.
    $$
    By analogy with \eqref{m64}, we have
     \bel{m66}
 {\mathcal E}^-_j(\eps) = 0, \quad \eps \in \re, \quad j \in \J_1,
 \ee
 if \eqref{ujf8} holds true. Moreover, if \eqref{f2}, \eqref{ujf1}, and \eqref{m1} are valid, we have $\eps_0^{\rm min} > 0$, and
 \bel{m67}
 {\mathcal E}^+_j(\eps) > 0, \quad \eps \in (-\infty, \eps_0^{\rm min}), \quad j \in \J_1,
 \ee
 by analogy with \eqref{m65}.

    \begin{theorem} \label{fth1}
  Let $\cs \subset \rd$ be a bounded domain such that ${\mathcal T} \neq 0$.
  Assume that:
  \begin{itemize}
  \item $w$ does not vanish identically on $\re$, and satisfies \eqref{ujf6} with $\beta = 2p + 1$, $p \in \N$, and
   \bel{f12}
  w(s) \geq 0, \quad s \in \re,
  \ee
   \item $\lambda_0$ satisfies \eqref{ujf1}, \eqref{ujf8}, \eqref{ujf14}, and \eqref{f4};
   \item the waveguide satisfies ``the thinness condition" \eqref{d4} with $\delta_0 = \frac{\eps_0^{\rm min}}{1+\eps_0^{\rm min}}$.
   \end{itemize}
   Then, again, we have
  \eqref{ujf17}.
  \end{theorem}
  For the proof of Theorem \ref{fth1}, we will need Lemma \ref{fl1} below. Let us recall that by \eqref{f1}, $h_{0,0}$ is simply the operator $-\frac{d^2}{ds^2}$, self-adjoint in ${\rm L}^2(\re)$, while $h_{0,0,\ell}$ is the Dirichlet realization of its restriction onto $(-\ell/2,\ell/2)$, $\ell \in (0,\infty)$.
  \begin{lemma} \label{fl1}
  Let $n \in \N$, $V_j : \re \times \Omega \to \re$, $j=1,\ldots,n$, be almost surely bounded ergodic potentials. Let $\rho_j$ be the IDS for the operator $h_{0,0} + nV_j$, $j=1,\ldots,n$, and $\rho$ be the IDS for the operator $h_{0,0} + \sum_{j=1}^n V_j$. Then we have
    \bel{f24}
  \rho(E) \leq \sum_{j=1}^n \rho_j(E), \quad E \in \re.
     \ee
     \end{lemma}
     {\em Remark}: Lemma \ref{fl1} admits an immediate extension to general multi-dimensional ergodic Schr\"odinger operators. The above formulation of the lemma is both convenient and sufficient for our purposes.
     \begin{proof}[Proof of Lemma \ref{fl1}]
    Let $E \in \re$. Then
    \bel{f25}
    \rho_j(E) = \lim_{\ell \to \infty} \ell^{-1}\,  {\rm Tr} \, \one_{(-\infty,E)}(h_{0,0,\ell} + nV_j), \quad j =1,\ldots,n,
    \ee
    \bel{f26}
    \rho(E) = \lim_{\ell \to \infty} \ell^{-1} \, {\rm Tr} \, \one_{(-\infty,E)}\left(h_{0,0,\ell} + \sum_{j=1}^n V_j\right).
    \ee
    On the other hand, a suitable version of the Weyl inequalities (see e.g. \cite[Eq.(125)]{RS4}) implies
    $$
    {\rm Tr} \, \one_{(-\infty,E)}\left(h_{0,0,\ell} + \sum_{j=1}^n V_j\right) = {\rm Tr} \, \one_{(-\infty,0)}\left(\sum_{j=1}^n\left(\frac{1}{n} h_{0,0,\ell} +  V_j - \frac{1}{n}E\right)\right) \leq
    $$
    \bel{f27}
    \sum_{j=1}^n {\rm Tr} \, \one_{(-\infty,0)}\left(\frac{1}{n} h_{0,0,\ell} +  V_j - \frac{1}{n}E\right) =
    \sum_{j=1}^n {\rm Tr} \, \one_{(-\infty,E)}\left(h_{0,0,\ell} +  nV_j\right).
    \ee
    Combining \eqref{f25}, \eqref{f26}, and \eqref{f27}, we arrive at \eqref{f24}.
    \end{proof}
  \begin{proof}[Proof of Theorem \ref{fth1}]
  By \eqref{d8} and \eqref{f13}, we immediately get
  \bel{f14}
\liminf_{E \downarrow 0} \frac{\ln{|\ln{N_{\gamma}(\mu_1 + E)}|}}{\ln{E}} \geq - \frac{1}{2}.
    \ee
    Let us obtain the corresponding upper bound. By \eqref{f12} and \eqref{f4}, we have
    \bel{f16}
    \gamma(s;\omega)^2 \geq \sum_{k \in \Z} \lambda_k(\omega)^2 w(s-k)^2, \quad s \in \re.
    \ee
    Applying the Cauchy-Schwarz inequality, we easily find that
    $$
    \dot{\gamma}(s;\omega)^2 = \left(\sum_{k \in \Z} \lambda_k(\omega) \dot{w}(s-k)\right)^2 =
    $$
    \bel{f17}
    = \left(\sum_{k \in \Z} \lambda_k(\omega) \dot{w}(s-k) \sum_{j \in \J_2} \one_{[-1/2,1/2)}(s-k-j)\right)^2 \leq n_2 \sum_{k \in \Z} \lambda_k(\omega)^2 \dot{w}(s-k)^2, \quad s \in \re.
    \ee
    Putting together \eqref{f16} and \eqref{f17}, we find that if $\eps \geq 0$, then
    \bel{f18}
    \Td \gamma(s;\omega)^2 - \eps \dot{\gamma}(s;\omega)^2 \geq \sum_{k \in \Z} \lambda_k(\omega)^2\left(\Td w(s-k)^2 - n_2 \eps \dot{w}(s-k)^2\right), \quad s \in \re.
    \ee
    Introduce the operator
    $$
    \tilde{h}_{\gamma, \eps} : = h_{0,0}+  \sum_{k \in \Z} \lambda_k(\omega)^2\left(\Td w(s-k)^2 - n_2 \eps \dot{w}(s-k)^2\right),
    $$
which is
    self-adjoint and $\Z$-ergodic in ${\rm L}^2(\re)$, and denote by $ \tilde{\nu}_{\gamma, \eps}$ its IDS. Then \eqref{f18} implies
    \bel{f21}
    \nu_{\gamma, \eps}(E) \leq \tilde{\nu}_{\gamma, \eps}(E), \quad E \in \re, \quad \eps \geq 0.
    \ee
    Next,
    \bel{f19a}
    \sum_{k \in \Z} \lambda_k(\omega)^2\left(\Td w(s-k)^2 - n_2 \eps \dot{w}(s-k)^2\right) = \sum_{j \in \J_1} \sum_{k \in \Z} \lambda_{k-j}(\omega)^2 v_{j, \eps}(s-k), \quad s \in \re,
    \ee
    the potentials $v_{j, \eps}$ being defined in \eqref{f19}. Denote by $ \tilde{\nu}_{\gamma, \eps. j}$, $j \in \J_1$, the IDS for the operator
    $$
    h_{0,0} + n_1 \sum_{k \in \Z} \lambda_{k-j}(\omega)^2 v_{j, \eps}(s-k),
    $$
which is
    self-adjoint and $\Z$-ergodic in ${\rm L}^2(\re)$. By \eqref{f19a}, and Lemma \ref{fl1},
    \bel{f20}
    \tilde{\nu}_{\gamma, \eps}(E) \leq \sum_{j \in \J_1} \tilde{\nu}_{\gamma, \eps, j}(E), \quad E \in \re, \quad \eps \in \re.
    \ee
    Arguing as in the proof of \eqref{ujf16}, we can show that \eqref{ujf8}, \eqref{m66}, and \eqref{m67}, imply
    \bel{f22}
    \limsup_{E \downarrow 0} \frac{\ln{|\ln{\tilde{\nu}_{\gamma, \eps, j}(E)|}}}{\ln{E}} \leq - \frac{1}{2}, \quad j \in \J_1, \quad \eps < \eps_0^{\rm min}.
    \ee
    Combining \eqref{d5}, \eqref{f21}, \eqref{f20}, and \eqref{f22}, we get
    \bel{f23}
\limsup_{E \downarrow 0} \frac{\ln{|\ln{N_{\gamma}(\mu_1 + E)}|}}{\ln{E}} \leq - \frac{1}{2}.
    \ee
    Putting together \eqref{f14} and \eqref{f23}, we arrive at \eqref{ujf17}.\\
    \end{proof}

{\bf Acknowledgements}. The authors gratefully acknowledge  the partial
support of the Chilean Scientific Foundation {\em Fondecyt} under Grants 1130591  and 1170816.
D. Krej\v{c}i\v{r}\'ik was also partially supported by the GACR Grant No.\ 18-08835S
and by FCT (Portugal) through Project PTDC/MAT-CAL/4334/2014.\\
A considerable part of this work has been done during W. Kirsch's visits to the {\em Pontificia Universidad Cat\'olica de Chile} in 2015 and 2016.
He thanks this university for hospitality. \\
Another substantial part of this work has been done during G. Raikov's visits to the University of Hagen,
Germany, the Czech Academy of Sciences, Prague, and the Institute of Mathematics, Bulgarian Academy of Sciences, Sofia.
He thanks these institutions for financial support and hospitality.


\vfill

{\sc Werner Kirsch}\\
Fakult\"at f\"ur Mathematik und Informatik\\
FernUniversit\"at in Hagen\\
Universit\"atsstrasse 1\\
D-58097 Hagen, Germany\\
E-mail: werner.kirsch@fernuni-hagen.de\\

{\sc David Krej\v{c}i\v{r}\'ik}\\
Department of Mathematics\\
 Faculty of Nuclear Sciences and Physical Engineering\\
Czech Technical University in Prague\\
 Trojanova 13\\
  12000 Prague 2, Czech Republic\\
E-mail: david.krejcirik@fjfi.cvut.cz\\

{\sc Georgi Raikov}\\
Facultad de Matem\'aticas\\
Pontificia Universidad Cat\'olica de Chile\\
Av. Vicu\~na Mackenna 4860\\ Santiago de Chile\\
E-mail: graikov@mat.uc.cl

\end{document}